\documentclass[a4paper, 12pt]{article}
\usepackage[utf8]{inputenc}
\usepackage[T1]{fontenc}
\usepackage{lmodern}
\usepackage{graphicx}
\usepackage{geometry}
\usepackage{array}
\usepackage{subfigure}
\usepackage[fleqn]{amsmath}
\usepackage{amsfonts, amssymb, amsthm, mathtools}
\usepackage{hyperref}
\usepackage{color}
\hypersetup{colorlinks = true, urlcolor = blue, bookmarksopen = true, linkcolor = blue, citecolor = blue}
\usepackage{multicol}

\usepackage[dvips]{epsfig}

\usepackage{dsfont}
\usepackage{stmaryrd}

\usepackage{verbatim}
\usepackage{tikz} 
\usepackage{enumerate}
\usepackage{pdfsync}
\usetikzlibrary{calc}
\usepackage{enumitem}

\usepackage[nottoc,notlot,notlof]{tocbibind}

\newcommand{\bb}{{\noindent}}
\newcommand{\R}{{\mathbb R}}
\newcommand{\C}{{\mathbb C}}

\newcommand{\E}{{\mathbb E}}

\newcommand{\PP}{{\mathbb P}}
\newcommand{\A}{{\mathcal A}}
\newcommand{\B}{{\mathcal B}}
\newcommand{\X}{{\mathcal X}}

\newcommand{\N}{{\mathbb N}}

\newcommand{\LL}{{\mathcal L}}

\newcommand{\I}{{\mathcal I}}
\newcommand{\J}{{\mathcal J}}
\newcommand{\K}{{\mathcal K}}

\newcommand{\1}{\mathds 1 }

\renewcommand{\d}{{\text{d}}}

\renewcommand*{\Re}{\operatorname{Re}}

\DeclareMathOperator{\tr}{tr}

\DeclareMathOperator{\Spec}{Spec}

\DeclareMathOperator{\diag}{diag}
\DeclareMathOperator{\Det}{Det}

\newtheorem{theorem}{Theorem}

\newtheorem{proposition}[theorem]{Proposition}
\newtheorem{corollary}[theorem]{Corollary}
\newtheorem{lemma}[theorem]{Lemma}

\theoremstyle{definition}
\newtheorem{remark}[theorem]{Remark}

\usepackage[english]{babel}

\begin{document}

\title{Necessary and sufficient conditions for the  existence \\ of $\alpha$-determinantal processes}
\author{Franck Maunoury \footnote{Université Pierre et Marie Curie / Telecom ParisTech, France. E-mail: franck.maunoury@upmc.fr}}

\date{Submitted: June 2014 - Revised: April 2015}

\maketitle

\thispagestyle{empty} 

\begin{abstract} 
\bb We  give necessary and sufficient conditions for existence and infinite divisibility of $\alpha$-determinantal processes.
For that purpose we use results on negative binomial and ordinary binomial multivariate distributions.
\end{abstract}

\noindent {\bf Keywords:} determinantal process, permanental process, $\alpha$-determinantal process, $\alpha$-permanental process, infinitely divisible, complete monotonicity, fermion process, boson process.\\

\noindent {\bf AMS 2010 subject classifications:} 60G55, 60E07.

\section{Introduction}

Several authors have already established necessary and sufficient conditions for existence of $\alpha$-determinantal processes.

\bb Macchi in~\cite{MR0380979} and Soshnikov in its survey paper \cite{Soshnikov determinantal random fields}  gave a necessary and sufficient condition for determinantal processes with self-adjoint kernels, which corresponds to the case $\alpha = -1$.

\bb The same condition has also been established in a different way by Hough, Krishnapur, Peres and Vir{\'a}g in \cite{MR2216966} in the case $\alpha = -1$.
They have also given a sufficient condition of existence in the case $\alpha=1$ and self-adjoint kernel.

\bb In the special case when the configurations are on a finite space, the paper of Vere-Jones \cite{Vere-Jones alpha-permanents applications} provides necessary and sufficient conditions for any value of $\alpha$.

\bb Finally, Shirai and Takahashi have given sufficient conditions for the existence of an $\alpha$-determinantal process for any values of $\alpha$. However, in the case $\alpha > 0$, their sufficient condition (Condition B) in~\cite{MR2018415} does not work for the following example: the space is reduced to a single point space and the reference measure $\lambda$ is a unit point mass. With their notations, the two kernels $K$ and $J_\alpha$ are respectively reduced to two real numbers $k$ and $j_\alpha$, with
\begin{eqnarray*}
	j_\alpha = \dfrac{k}{1+ \alpha k}
\end{eqnarray*} 
We can choose $\alpha>0$ and $k<0$ such that $j_\alpha > 0$. Under these assumptions, Condition B is fulfilled but the obtained point process has a negative correlation function ($\rho_1(x) = k$), which has to be excluded, since a correlation function is an almost everywhere non-negative function.
\smallskip

\bb We are going to strengthen Condition B of Shirai and Takahashi and obtain a necessary and sufficient condition in the case $\alpha >0$. This is presented in Theorem~\ref{th CNS alpha-permanental}.

\bb Besides, in the case $\alpha<0$, we extend the result of Shirai and Takahashi to the case of non self-adjoint kernels and show that the obtained condition is also necessary (Theorems~\ref{th CNS alpha-determinantal si I+alpha K inversible} and \ref{th CNS alpha-determinantal si I+alpha K_Lambda not invertible}). Moreover, we show that $-1/\alpha$ is necesserely an integer. This has been noticed by Vere-Jones in \cite{Vere-Jones permanent determinant generalization} in the case of configurations on a finite space. 

\bb We also give a necessary and sufficient condition for the infinite divisibility of an $\alpha$-determinantal process for all values of $\alpha$.
\smallskip

\bb The main results are presented in Section~\ref{section main theorems}. Section \ref{section preliminaries} introduces the needed notation. 
In Section~\ref{section Fredhom det expansion}, we write a multivariate version of a Shirai and Takahashi formulae on Fredholm determinant expansion. 
Sections~\ref{section alpha-permanental} and~\ref{section alpha-determinantal} present the proofs of the results concerning respectively the cases $\alpha > 0$ and $\alpha<0$.
The proofs concerning infinite divisibility are presented in Section \ref{section infinite divisibility}.

\section{Preliminaries}
\label{section preliminaries}

Let $E$ be a locally compact Polish space. A locally finite configuration on $E$ is an integer-valued positive Radon measure on $E$. It can also be identified with a set $\{(M, \alpha_M) : M \in F \}$, where $F$ is a countable subset of $E$ with no accumulation points (i.e. a discrete subset of $E$) and, for each point in $F$, $\alpha_M$ is a non-null integer that corresponds to the multiplicity of the point $M$ (M is a multiple point if $\alpha_M \geq 2$).

\bb Let $\lambda$ be a Radon measure on $E$. 
Let $\X$ be  the space of the locally finite configurations  of $E$. The space $\X$ is endowed with the vague topology of measures, i.e. the smallest topology such that, for every real continuous function $f$ with  compact support,  defined on $E$ , the mapping
\begin{eqnarray*}
  \X \ni \xi \mapsto \left\langle f,\xi \right\rangle =
  \sum_{x\in \xi}f(x)=\int{f \d\xi}
\end{eqnarray*}
is continuous. Details on the topology of the configuration space can be found in~\cite{MR99d:58179}.

\bb We denote by $\B(\X)$ the corresponding $\sigma$-algebra.
 A point process on $E$ is a random variable with values in $\X$. We do not restrict ourselves to simple point processes, as the configurations in $\X$ can have multiple points. \\

\bb For a $n \times n$ matrix $A = (a_{[i j})_{1 \leq i,j \leq n}$, set:
\begin{eqnarray*}
	\det\nolimits_\alpha A = \sum_{\sigma \in \Sigma_n} \alpha^{n- \nu(\sigma)} \prod_{i=1}^n a_{i \sigma(i)}
\end{eqnarray*}
where $\Sigma_n$ is the set of all permutations on $\{1, \dots , n\}$ and $\nu(\sigma)$ is the number of cycles of the permutation $\sigma$.\\

For a relatively compact set $\Lambda \subset E$, the Janossy densities of a point process $\xi$ w.r.t. a Radon measure $\lambda$ are functions (when they exist) $j^\Lambda_n : E^n \rightarrow [0, \infty)$ for $n \in N$, such that
\begin{eqnarray*}
	& j^\Lambda_n (x_1, \dots, x_n) = n! \,\, \PP(\xi(\Lambda) = n) \, \pi^\Lambda_n (x_1, \dots, x_n)\\
	& j^\Lambda_0 (\emptyset) = \PP(\xi(\Lambda) = 0),
\end{eqnarray*}
where $\pi^\Lambda_n$ is the density with respect to $\lambda^{\otimes n}$ of the ordered set $(x_1, \dots, x_n)$, obtained by first sampling $\xi$, given that there are n points in $\Lambda$, then choosing uniformly an order between the points.

For $\Lambda_1, \dots, \Lambda_n$ disjoint subsets included in $\Lambda$, $\int_{\Lambda_1 \times \dots \times \Lambda_n} j_n^\Lambda (x_1, \dots, x_n) \lambda(dx_1) \dots \lambda(dx_n)$ is the probability that there is exactly one point in each subset $\Lambda_i$ ($1 \leq i \leq n$), and no other point elsewhere.

We recall that we have the following formula, for a non-negative mesurable function $f$ with support in a relatively compact set $\Lambda \subset E$:
\begin{eqnarray*}
	\E(f(\xi)) = f(\emptyset) \, j_0^\Lambda (\emptyset) + \sum_{n=1}^\infty \dfrac{1}{n!} \int_{\Lambda^n} f(x_1, \dots, x_n) \, j^\Lambda_n (x_1, \dots, x_n) \lambda(dx_1) \dots \lambda(dx_n).\\
\end{eqnarray*}

For $n \in \N$ and $a \in \R$, we denote $a^{(n)} = \prod_{i=0}^{n-1} (a - i)$.

The correlation functions (also called joint intensities) of a point process $\xi$ w.r.t. a Radon measure $\lambda$ are functions (when they exist) $\rho_n : E^n \rightarrow [0, \infty)$ for $n \geq 1$, such that for any family of mutually disjoint relatively compact subsets $\Lambda_1, \dots, \Lambda_d$ of $E$ and for any non-null integers $n_1, \dots, n_d$ such that $n_1 + \dots + n_d = n$, we have
\begin{eqnarray*}
	\E \left(\prod_{i=1}^d \xi(\Lambda_i)^{(n_i)} \right) = \int_{\Lambda_1^{n_1} \times  \dots \times \Lambda_d^{n_d}} \rho_n(x_1, \dots, x_n) \lambda(dx_1), \dots, \lambda(dx_n).
\end{eqnarray*}

Intuitively, for a simple point process, $\rho_n(x_1, \dots, x_n) \lambda(dx_1) \dots \lambda(dx_n)$ is the infinitesimal probability that there is at least one point in the vicinity of each $x_i$ (each vicinity having an infinitesimal volume $\lambda(dx_i)$ around $x_i$), $1 \leq i \leq n$.\\

\bb Let $\alpha$ be a real number and  $K$ a kernel from $E^2$ to $\R$ or $\C$. 
 An  $\alpha$-determinantal point process, with kernel $K$ with respect to  $\lambda$ (also called $\alpha$-permanental point process) 
 is defined, when it exists,  as a point process with the following  correlation functions $\rho_n, n \in \N$ with respect to $\lambda$:
\begin{eqnarray*}
	\rho_n(x_1, \dots, x_n) = \det\nolimits_\alpha (K(x_i,x_j))_{1 \leq i,j \leq n}.
\end{eqnarray*}
We denote by $\mu_{\alpha,K,\lambda}$  the probability distribution of such a point process.

\bb  We exclude the case  of  a point process almost surely reduced to the empty configuration.
\smallskip

\bb The case $\alpha = -1$ corresponds to a determinantal process and the case $\alpha = 1$ to a permanental process. 
The case $\alpha=0$ corresponds to the  Poisson point process. We suppose in the following that $\alpha \neq 0$.\\

\bb We will always assume that the kernel $K$ defines a locally trace class integral operator $\K$ on $L^2(E, \lambda)$. Under this assumption, one obtains an equivalent definition for the $\alpha$-determinantal process, using the following Laplace functional formula:
\begin{eqnarray}
\label{alpha-det Laplace functional}
\E_{\mu_{\alpha,K,\lambda}} \left[\exp \left(-\int_E f d\xi \right) \right] = \Det \left(\I + \alpha \K [1-e^{-f}] \right)^{-1/\alpha}
\end{eqnarray}
where $f$ is a compactly-supported non-negative function on $E$, $\K [1-e^{-f}]$ stands for $\sqrt{1-e^{-f}} \K \sqrt{1-e^{-f}}$, $\I$ is the identity operator on $L^2(E,\lambda)$ and $\Det$ is the Fredholm determinant. Details on the link between the correlation function and the Laplace functional of an $\alpha$-determinantal process can be found in the chapter 4 of~\cite{MR2018415}. Some explanations and useful formula on the Fredholm determinant are given in chapter 2.1 of~\cite{MR2018415}.\\
\smallskip

\bb For a subset $\Lambda \subset E$, set: $\K_\Lambda = p_\Lambda \K p_\Lambda$, where $p_\Lambda$ is the orthogonal projection operator from $L^2(E,\lambda)$ to the subspace $L^2(\Lambda,\lambda)$.\\

\bb For two subsets $\Lambda, \Lambda' \subset E$, set: $\K_{\Lambda \Lambda'} = p_\Lambda \K p_{\Lambda'}$, and denote by $K_{\Lambda \Lambda'}$ its kernel. We have for any $x,y \in E$, $K_{\Lambda \Lambda'} (x,y) = \1_\Lambda (x) \, \1_{\Lambda'} (y) \, K(x,y)$. \\

 \bb When $\I + \alpha \K$ (resp. $\I + \alpha \K_\Lambda$) is invertible, $\J_\alpha$ (resp. $\J_\alpha^\Lambda$) is the integral operator defined by:
	$\J_\alpha = \K (\I + \alpha \K)^{-1} \text{ (resp. } \J_\alpha^\Lambda = \K_\Lambda (\I + \alpha \K_\Lambda)^{-1} \text{)}$
 and we denote by $J_\alpha$ (resp. $J_\alpha^\Lambda$) its kernel. Note that $\J_\alpha^\Lambda$ is not the orthogonal projection of $\J_\alpha$ on $L^2(\Lambda, \lambda)$. \\

\section{Main results}
\label{section main theorems}

\begin{theorem}
\label{th CNS alpha-permanental}
For $\alpha > 0$, there exists an $\alpha$-permanental process with kernel $K$ iff:
\begin{itemize}
\item $\Det(\I+\alpha \K_\Lambda) \geq 1$, for any compact set $\Lambda \subset E$
\item $\det_\alpha (J_\alpha^\Lambda (x_i, x_j))_{1 \leq i,j \leq n} \geq 0$, for any $n \in \N$, any compact set $\Lambda \subset E$ and any $\lambda^{\otimes n}$-a.e. $(x_1, \dots, x_n) \in \Lambda^n$.\\
\end{itemize}

\end{theorem}

\begin{remark}
	{\rm Even when $E$ is a finite set, note that the second condition of Theorem~\ref{th CNS alpha-permanental} consists in an infinite number of computations.  Finding a simpler condition, that could be checked in a finite number of steps is still an open problem.}\\
\end{remark}

\begin{theorem}
\label{th Spec alpha-permanental}
For $\alpha > 0$, if an $\alpha$-permanental process with kernel $K$ exists, then:
\begin{eqnarray*}
	\Spec \K_\Lambda \subset \{ z \in \C : \Re z > - \dfrac{1}{2 \alpha} \} \text{ , for any compact set } \Lambda \subset E.\\
\end{eqnarray*}

\end{theorem}

We remark that this condition is equivalent to
\begin{eqnarray*}
	\Spec \J_\alpha^\Lambda \subset \{ z \in \C : |z| < \dfrac{1}{\alpha} \} \text{ , for any compact set } \Lambda \subset E
\end{eqnarray*}

\begin{theorem}
\label{th CNS alpha-determinantal si I+alpha K inversible}
For $\alpha < 0$ and $\K$ an integral operator such that $\I+\alpha \K_\Lambda$ is invertible, for any compact set $\Lambda \subset E$, an $\alpha$-determinantal process with kernel $K$ exists iff the two following conditions are fulfilled:
\begin{itemize}
\item[(i)] $-1/\alpha \in \N$
\item[(ii)] $\det (J_\alpha^\Lambda (x_i, x_j))_{1 \leq i,j \leq n} \geq 0$, for any $n \in \N$, any compact set $\Lambda \subset E$ and any $\lambda^{\otimes n}$-a.e. $(x_1, \dots, x_n) \in \Lambda^n$.
\end{itemize}

\end{theorem}

The arguments developed in the proof of Theorem~\ref{th CNS alpha-determinantal si I+alpha K inversible} shows that actually $(ii) \implies (i)$. Consequently, Condition $(ii)$ is itself a necessary and sufficient condition. It also implies that $\Det(\I + \beta \K_\Lambda)>0 $ for any $\beta \in [\alpha,0]$ and any compact $\Lambda \subset E$.\\

\begin{theorem}
\label{th CNS alpha-determinantal si I+alpha K_Lambda not invertible}
For $\alpha < 0$ and $\K$ an integral operator such that for some compact set $\Lambda_0 \subset E$, $\I+\alpha \K_{\Lambda_0}$ is not invertible, an $\alpha$-determinantal process with kernel $K$ exists iff:
\begin{itemize}
\item[(i')] $-1/\alpha \in \N$
\item[(ii')] $\det (J_\beta^\Lambda (x_i, x_j))_{1 \leq i,j \leq n} \geq 0$, for any $n \in \N$, any $\beta \in (\alpha,0)$, any compact set $\Lambda \subset E$ and any $\lambda^{\otimes n}$-a.e. $(x_1, \dots, x_n) \in \Lambda^n$.
\end{itemize}

\end{theorem}

As in Theorem~\ref{th CNS alpha-determinantal si I+alpha K inversible}, we also have $(ii') \implies (i')$ and Condition $(ii')$ is itself a necessary and sufficient condition. \\

Note that $\I+\alpha \K_{\Lambda_0}$ is not invertible if and only if there is almost surely at least one point in $\Lambda_0$.

\begin{corollary}
\label{corol equiv exist -1/m det and det process}
	For $m$ a positive integer, the existence of a $(-1/m)$-determinantal process with kernel $K$ is equivalent to the existence of a determinantal process with the kernel $\dfrac{K}{m}$.\\
\end{corollary}

\begin{corollary}
\label{corol CNS alpha-determinantal si K selfadjoint}
For $\alpha < 0$ and $\K$ a self-adjoint operator, an $\alpha$-determinantal process with kernel $K$ exists iff:
\begin{itemize}
\item $-1/\alpha \in \N$
\item $\Spec \K \subset [0,-1/\alpha]$
\end{itemize}

\end{corollary}

This result is well known in the case $\alpha = -1$ (see for example Hough, Krishnapur, Peres and Vir{\'a}g in \cite{MR2216966}).

The sufficient part of this necessary and sufficient condition corresponds to condition A in \cite{MR2018415} of Shirai and Takahashi.\\

\begin{theorem}
\label{th infinie divisibilité alpha-déterminantal}
	For $\alpha < 0$, an $\alpha$-determinantal process in never infinitely divisible.\\
\end{theorem}

\begin{theorem}
\label{th infinie divisibilité général alpha-permanental}
	For $\alpha > 0$, an $\alpha$-determinantal process is infinitely divisible iff
	\begin{itemize}
\item $\Det(\I+\alpha \K_\Lambda) \geq 1$, for any compact set $\Lambda \subset E$
\item $ \sum_{\sigma \in \Sigma_n : \nu(\sigma) = 1} \prod_{i=1}^n J^\Lambda_\alpha(x_i, x_{\sigma(i)}) \geq 0$, for any $n \in \N$, any compact set $\Lambda \subset E$ and $\lambda^{\otimes n}$-a.e. $(x_1, \dots, x_n) \in \Lambda^n$.
\end{itemize}
	
\end{theorem}

This theorem gives a more general condition for infinite-divisibility of an $\alpha$-permanental process than the condition given by Shirai and Takahashi in~\cite{MR2018415}.\\

\begin{theorem}
\label{th infinie divisibilité self-adjoint alpha permanental}
	For $\K$ a a real symmetric locally trace class operator and $\alpha > 0$, an $\alpha$-permanental process is infinitely divisible iff
	\begin{itemize}
\item $\Det(\I+\alpha \K_\Lambda) \geq 1$, for any compact set $\Lambda \subset E$
\item $ J^\Lambda_\alpha(x_1, x_2) \dots J^\Lambda_\alpha(x_{n-1}, x_n) J^\Lambda_\alpha(x_n, x_1) \geq 0$, for any $n \in \N$, any compact set $\Lambda \subset E$ and $\lambda^{\otimes n}$-a.e. $(x_1, \dots, x_n) \in \Lambda^n$.\\
\end{itemize}
	
\end{theorem}

Following Griffith and Milne's remark in~\cite{Griffiths and Milne}, when an $\alpha$-permanental process with kernel $K$ exists and is infinitely divisible, we can replace $J^\alpha_\Lambda$ by $|J^\alpha_\Lambda|$ and obtain an $\alpha$-permanental process with the same probability distribution.\\

\begin{remark}
\label{rq CNS alpha-permanental}
In Theorem~\ref{th CNS alpha-permanental},~\ref{th infinie divisibilité général alpha-permanental} and~\ref{th infinie divisibilité self-adjoint alpha permanental} , the condition
\begin{eqnarray*}
	\Det(\I+\alpha \K_\Lambda) \geq 1 \text{, for any compact set } \Lambda \subset E
\end{eqnarray*}

can be replaced by
\begin{eqnarray*}
	\Det(\I+\alpha \K_\Lambda) > 0 \text{, for any compact set } \Lambda \subset E.
\end{eqnarray*}

\end{remark}

\section{Fredholm determinant expansion}
\label{section Fredhom det expansion}

In \cite{MR2018415}, Shirai and Takahashi have proved the following formula
\begin{eqnarray}
\label{formula dec det Fredholm}
\Det (\I - \alpha z \K)^{-1/\alpha} = \sum_{n=0}^\infty \dfrac{z^n}{n!} \int_{E^n} \det\nolimits_\alpha (K(x_i,x_j))_{1 \leq i,j \leq n} \lambda(dx_1) \dots \lambda(dx_n)
\end{eqnarray}
for a trace class integral operator $\K$ with kernel $K$ and for $z \in \C$ such that $\| \alpha z \K \| < 1$. In the case where the space $E$ is finite, this formula is also given by Shirai in~\cite{Shirai}.\\

As $z \mapsto \Det (\I - \alpha z \K)$ is analytic on $\C$ and $z \mapsto z^{-1/\alpha}$ is analytic on $\C^*$, we obtain that $z \mapsto \Det (\I - \alpha z \K_{\Lambda,\alpha})^{-1/\alpha}$ is analytic on $\{z \in \C : \I-\alpha z \K_{\Lambda,\alpha} \text{ invertible} \}$.

Therefore, the formula can be extended to the open disc $D$, centered in $0$ with radius $R = \sup \{ r \in \R_+ : \forall z \in \C, |z| < r \Rightarrow \I - \alpha z \K \text{ is invertible}\}$.

$D$ is the open disc of center 0 and radius $1/ \| \alpha \K \|$, if the operator $\K$ is self-adjoint, but it can be larger if $\K$ is not self-adjoint. \\   

As remarked by Shirai and Takahashi, the formula $\eqref{formula dec det Fredholm}$ is valid for any $z\in \C$ if $-1/\alpha \in \N$.\\

The following proposition extends $\eqref{formula dec det Fredholm}$ to a multivariate case.

\begin{proposition}
	Let $\Lambda \subset E$ be a relatively compact set, $\Lambda_1, \dots \Lambda_d$ mutually disjoint subsets of $\Lambda$ and $\K$ a locally trace class integral operator with kernel $K$. \\
	We have the following formula
\begin{align}
\label{formula dec det Fredholm multivariate}
	& \Det \left(\I - \alpha \sum_{k=1}^d z_k \, \K_{\Lambda_k \Lambda}\right)^{-1/\alpha}\nonumber \\
	& \qquad = \sum_{n_1, \dots, n_d=0}^\infty \, \left(\prod_{k=1}^d \dfrac{z_k^{n_k}}{n_k!}\right) \int_{\Lambda_1^{n_1} \times \dots \times \Lambda_d^{n_d}} \det\nolimits_\alpha (K(x_i,x_j))_{1 \leq i,j \leq n} \, \lambda(dx_1) \dots \lambda(dx_n)
\end{align}	
for any $z_1, \dots, z_d \in \C$, such that $\I - \alpha \gamma \sum_{k=1}^d z_k \K_{\Lambda_k \Lambda}$ is invertible for any complex number $\gamma$ satisfying $|\gamma| < 1$ ($n$ denotes $n_1 + \dots + n_d$).
\end{proposition}

\begin{proof}
	We apply the formula~\eqref{formula dec det Fredholm} to the class trace operator $\sum_{k=1}^d z_k \K_{\Lambda_k \Lambda}$ and we use the multilinearity property of the $\alpha$-determinant of a matrix with respect to its rows.\\
	
We obtain
\begin{align*}
	\Det & \left(\I - \alpha \sum_{k=1}^d z_k \K_{\Lambda_k \Lambda}\right)^{-1/\alpha} \\
	& =  \sum_{n=0}^\infty \dfrac{1}{n!} \int_{E^n} \det\nolimits_\alpha \left(\sum_{k=1}^d z_k K_{\Lambda_k \Lambda}(x_i,x_j)\right)_{1 \leq i,j \leq n} \lambda(dx_1) \dots \lambda(dx_n)\\
	& = \sum_{n=0}^\infty \dfrac{1}{n!} \, \int_{E^n} \sum_{k_1, \dots k_n = 1}^d \det\nolimits_\alpha \left(z_{k_i} \1_{\Lambda_{k_i}} (x_i) \1_\Lambda (x_j) K(x_i,x_j)\right)_{1 \leq i,j \leq n} \lambda(dx_1) \dots \lambda(dx_n)\\
	& = \sum_{n=0}^\infty \dfrac{1}{n!} \sum_{k_1, \dots k_n = 1}^d \int_{\Lambda_{k_1} \times \dots \times \Lambda_{k_n}} \det\nolimits_\alpha \left(z_{k_i} K(x_i,x_j)\right)_{1 \leq i,j \leq n} \lambda(dx_1) \dots \lambda(dx_n)\\
	& = \sum_{n=0}^\infty \dfrac{1}{n!} \sum_{k_1, \dots k_n = 1}^d \left(\prod_{i=1}^n z_{k_i} \right) \int_{\Lambda_{k_1} \times \dots \times \Lambda_{k_n}} \det\nolimits_\alpha \left(K(x_i,x_j)\right)_{1 \leq i,j \leq n} \lambda(dx_1) \dots \lambda(dx_n)\\
\end{align*}
where we have used the fact that $K_{\Lambda_k \Lambda} (x_i, x_j) = \1_{\Lambda_k} (x_i) \, \1_\Lambda (x_j) K(x_i , x_j)$ for the equality between the first and the second line.

As the value of the $\alpha$-determinant of a matrix is unchanged by simultaneous interchange of its rows and its columns, the product $z_1^{n_1} \dots z_d^{n_d}$ where $n_1 + \dots n_d = n$, will be repeated $\binom{n}{n_1 \dots n_d}$ times. This gives the desired formula.

\end{proof}

For a relatively compact set $\Lambda \subset E$ and $\Lambda_1, \dots, \Lambda_d$ mutually disjoint subsets of $\Lambda$, the computation of the Laplace functional of an $\alpha$-determinantal process for the function $f : (z_1, \dots , z_d) \mapsto - \sum_{k=1}^d (\log z_k) \1_{\Lambda_k}$, with $z_1, \dots, z_d \in (0,1]$ gives thanks to \eqref{alpha-det Laplace functional}:
\begin{eqnarray}
\label{formula pgf finite dim determinantal}
\E_{\mu_{\alpha,K,\lambda}} \left[\prod_{k=1}^d z_k^{\xi(\Lambda_k)}\right] = \Det \left(\I + \alpha \sum_{k=1}^d (1-z_k) \, \K_{\Lambda_k \Lambda}\right)^{-1/\alpha}
\end{eqnarray}
which is the probability generating function (p.g.f.) of the finite-dimensional random vector $(\xi(\Lambda_1), \dots, \xi(\Lambda_d))$.\\

For $\alpha<0$, the formula~\eqref{formula pgf finite dim determinantal} reminds the multivariate binomial distribution p.g.f. and for $\alpha>0$, the multivariate negative binomial distribution p.g.f., given by Vere-Jones in~\cite{Vere-Jones alpha-permanents applications}, in the special case where the space $E$ is finite.

\section{$\alpha$- permanental process ($\alpha > 0$)}
\label{section alpha-permanental}

\begin{proof} [Proof of Theorem~\ref{th CNS alpha-permanental}]

We first prove that the conditions are necessary. We suppose that there exists an $\alpha$-permanental process with $\alpha >0$, kernel $K$ defining the locally trace class integral operator $\K$.\\

By taking $d=1$ in the formula \eqref{formula pgf finite dim determinantal}, we have
\begin{eqnarray*}
	\E_{\mu_{\alpha,K,\lambda}} \left(z^{\xi(\Lambda)}\right) = \Det \left(\I + \alpha (1-z) \, \K_\Lambda\right)^{-1/\alpha}
\end{eqnarray*} 
for any compact set $\Lambda \subset E$ and $z \in (0,1]$.\\

Thus, $\Det(\I + \alpha (1-z) \K_\Lambda) \geq 1$ for $z \in (0,1]$. By continuity (as $z \mapsto \Det(\I + (1-z) \K_\Lambda)$ is indeed analytic on $\C$), we obtain that $\Det(\I + \alpha \K_\Lambda) \geq 1$, which is the first condition. This implies that for any compact set $\Lambda \subset E$, $\I + \alpha \K_\Lambda$ is invertible. Hence $\J_\alpha^\Lambda$ exists and we have, for any non-negative function $f$, with compact support included in $\Lambda$
\begin{align}
	& \E_{\mu_{\alpha,K,\lambda}} \left(\prod_{x \in \xi} e^{-f(x)} \right) = \Det (\I + \alpha \K [1-e^{-f}])^{-1/\alpha} \nonumber \\
	& = \Det (\I + \alpha \K_\Lambda (1-e^{-f}))^{-1/\alpha} \nonumber \\
	& = \Det (\I + \alpha \K_\Lambda)^{-1/\alpha} \Det (\I - \alpha \J_\alpha^\Lambda e^{-f})^{-1/\alpha} \nonumber \\
	\label{formula laplace janossy}
	& = \Det(\I + \alpha \K_\Lambda)^{-1/\alpha} \sum_{n=0}^\infty \dfrac{1}{n!} \int_{\Lambda^n} \left(\prod_{i=1}^n e^{-f(x_i)} \right) \det\nolimits_\alpha (J_\alpha^\Lambda(x_i,x_j))_{1 \leq i,j \leq n} \lambda(dx_1) \dots \lambda(dx_n)
\end{align}
where we have used for the equality between the first and the second line the fact that $\Det(\I + \A \B) = \Det(\I + \B \A)$, for any trace class operator $\A$,  and any bounded operator $\B$.

As the Laplace functional defines a.e. uniquely the Janossy density of a point process, one obtains: 
\begin{eqnarray*}
	\det\nolimits_\alpha (J_\alpha^\Lambda(x_i,x_j))_{1 \leq i,j \leq n} \geq 0 \text{ } \lambda^{\otimes n} \text{-a.e. } (x_1, \dots, x_n) \in E^n
\end{eqnarray*}

$j_{\alpha,n}^\Lambda(x_1, \dots, x_n) = \Det(\I + \alpha \K_\Lambda)^{-1/\alpha} \det\nolimits_\alpha (J_\alpha^\Lambda(x_i,x_j))_{1 \leq i,j \leq n}$ is the Janossy density.\\

Conversely, if we assume $\Det(\I + \alpha \K_\Lambda)^{-1/\alpha} > 0$ and $\det\nolimits_\alpha (J_\alpha^\Lambda(x_i,x_j))_{1 \leq i,j \leq n} \geq 0$ for any $n \in \N$, any compact set $\Lambda \subset E$ and any $\lambda^{\otimes n}$-a.e. $(x_1, \dots, x_n) \in \Lambda^n$, the Janossy density will be correctly defined and, on any compact set $\Lambda$, we get the existence of a point process $\xi_\Lambda$ with kernel $K_\Lambda$ (see Proposition 5.3.II. in~\cite{MR1950431} - here the normalization condition is automatic by chosing $f=0$ in~\eqref{formula laplace janossy}).

The restriction of a point process $\eta$, defined on $\Lambda' \subset E$, to a subspace $\Lambda \subset \Lambda'$ is the point process denoted $\eta|_\Lambda$, obtained by keeping the points in $\Lambda$ and deleting the points in $\Lambda' \backslash \Lambda$.

For any compact sets $\Lambda, \Lambda' \subset E$, such that $\Lambda \subset \Lambda'$, $\xi_\Lambda$ and $\xi_{\Lambda'}|_\Lambda$ have the same Laplace functional, because we have for any non-negative function f, with compact support included in $\Lambda$:
\begin{eqnarray*}
	\E\left(\exp\left(-\int_\Lambda f d\xi_{\Lambda'}|_\Lambda \right) \right) = \Det (\I + \alpha \K_{\Lambda'} [1-e^{-f}])^{-1/\alpha} \\ = \Det (\I + \alpha \K_{\Lambda} [1-e^{-f}])^{-1/\alpha} \\ = \E\left(\exp\left(-\int_\Lambda f d\xi_\Lambda \right) \right).
\end{eqnarray*}
Therefore, $\xi_\Lambda$ and $\xi_{\Lambda'}|_\Lambda$ have the same probability distribution. We say that the family $(\LL(\xi_\Lambda))$, $\Lambda$ compact set included in $E$, is consistent. 

Then we can obtain a point process on the complete space $E$ by the Kolmogorov existence theorem for point processes (see Theorem 9.2.X in~\cite{Daley Vere-Jones vol II} with $P_k(A_1, \dots, A_k; n_1, \dots, n_k) = \PP\left(\xi_{\cup_{i=1}^k A_i}(A_1)=n_1, \dots, \xi_{\cup_{i=1}^k A_i}(A_k)=n_k \right)$: as $\xi_{\cup_{i=1}^k A_i}$ is a point process, it follows that the properties (i), (iii), (iv) are fulfilled ; (ii) is fulfilled because the family $(\LL(\xi_\Lambda))$, $\Lambda$ compact set included in $E$, is consistent).

As we used, in this second part of the proof, only the fact that $\Det(\I + \alpha \K_\Lambda)^{-1/\alpha} > 0$ (instead of $\Det(\I + \alpha \K_\Lambda)^{-1/\alpha} \geq 1$), the assertion in remark~\ref{rq CNS alpha-permanental} is also proved.\\
\end{proof}

\begin{proof} [Proof of Theorem~\ref{th Spec alpha-permanental}]

We suppose there exists an $\alpha$-permanental process with $\alpha >0$, kernel $K$ defining the locally trace class integral operator $\K$.\\

Then, following the proof of the preceding theorem, we get that, for all $z \in [0,1]$
\begin{eqnarray*}
	\Det(\I + \alpha (1-z)\K_\Lambda) = \Det(\I + \alpha \K_\Lambda) \Det(\I - \alpha z\J_\alpha^\Lambda) > 0.
\end{eqnarray*}

As the power series of $\Det(\I - \alpha z\J_\alpha^\Lambda)^{-1/\alpha}$ has all its terms non-negative, 
\begin{eqnarray*}
|(\Det(\I - \alpha z \J_\Lambda^\alpha)^{-1/\alpha}| \leq (\Det(\I - \alpha \, |z| \, \J_\Lambda^\alpha)^{-1/\alpha}.
\end{eqnarray*}

If $z_0$ is a complex number with minimum modulus such that $(\Det(\I - \alpha z_0 \J_\Lambda^\alpha) = 0$, by analycity of $z \mapsto \Det(\I - \alpha z\J_\alpha^\Lambda)$ on $\C$ and $z \mapsto z^{-1}$ on $\C^*$, $\Det(\I - \alpha z\J_\alpha^\Lambda)^{-1/\alpha}$ converges for $|z| < |z_0|$ and diverges for $z=z_0$. Thus the series diverges in $z=|z_0|$ and $|z_0| > 1$. This means that the series converges for $|z| \leq 1$ thus, in this case, $\Det(\I - \alpha z \J_\alpha^\Lambda) > 0$.\\

This implies the necessary condition: $\Spec \J_\alpha^\Lambda \subset \{ z \in \C : |z| < \dfrac{1}{\alpha} \}$.\\

As $\nu$ eigenvalue of $\K$ is equivalent to $\dfrac{\nu}{1+\alpha \nu}$ eigenvalue of $\J$, and as, $\K$ and $\J$ being compact operators, their non-null spectral values are their eigenvalues, we get the other equivalent necessary condition:
\begin{eqnarray*}
\Spec \K_\Lambda \subset \{ z \in \C : \Re z > - \dfrac{1}{2 \alpha} \}.
\end{eqnarray*}

\end{proof}

\section{$\alpha$- determinantal process ($\alpha < 0$)}
\label{section alpha-determinantal}

We recall the following remark, already made for example in \cite{MR2216966}.

\begin{remark}
\label{rq def p.s. kernels det_alpha}
If we define kernels only $\lambda^{\otimes 2}$-almost everywhere, there can be problems when we consider only the diagonal terms, as $\lambda^{\otimes 2} \{ (x,x): x \in \Lambda \} =0$. For example, in the formula
\begin{eqnarray*}
	\tr K_\Lambda = \int_\Lambda K(x,x) \lambda(dx),
\end{eqnarray*}
$\tr K_\Lambda$ is not uniquely defined.
To avoid this problem, we write the kernel $K_\Lambda$ as follows:
\begin{eqnarray*}
	K_\Lambda (x,y) = \sum_{k=0}^\infty a_k \varphi_k(x) \overline{\psi_k}(y)
\end{eqnarray*}
where $(\varphi_k)_{k \in \N}$, $(\psi_k)_{k \in \N}$ are orthonormal basis in $L^2(\Lambda, \lambda)$ and $(a_k)_{k \in \N}$ is a sequence of non-negative real number, which are the singular values of the operator $\K_\Lambda$.

The functions $\varphi_k$ and $\psi_k$, $k \in \N$, are defined $\lambda$-almost everywhere, but this gives then a unique value for the expression of type
\begin{eqnarray*}
	\int_{\Lambda^n} F(K(x_i, x_j)_{1 \leq i,j \leq n}) \, G(x_1, \dots, x_n) \lambda(dx_1) \dots \lambda(dx_n)
\end{eqnarray*}
where $F$ is an arbitrary complex function from $\C^{n^2}$ and $G$ is an arbitrary complex function from $\Lambda^n$.

With this remark, the quantities that appear with $F = \det\nolimits_\alpha$ are well defined.\\

\end{remark}

\begin{lemma}
\label{lemma p.s. kernel}
Let $K$ be a kernel defined as in Remark~\ref{rq def p.s. kernels det_alpha} and defining a trace class integral operator $\K$ on $L^2(\Lambda, \lambda)$, where $\Lambda$ is a non-$\lambda$-null compact set included in the locally compact Polish space $E$, $\lambda$ be a Radon measure, $n$ an integer and $\alpha$ a real number. Let $F$ be a continuous fonction from $\C^{n^2}$ to $\C$. The three following assertions are equivalent
\begin{align*}
& \text{(i) } F (K(x_i,x_j)_{1 \leq i,j \leq n}) \geq 0 \text{ } \lambda^{\otimes n}-a.e. (x_1, \dots, x_n) \in \Lambda^n  \\
& \text{(ii) there exists a set } \Lambda' \subset \Lambda \text{ such that } \lambda(\Lambda \backslash \Lambda') = 0 \text{ and } F ((K(x_i,x_j))_{1 \leq i,j \leq n}) \geq 0 \\
& \,\,\,\,\,\,\,\, \text{ for any }(x_1, \dots, x_n) \in (\Lambda')^n  \\
& \text{(iii) there exists a version of } K \text{ such that } F ((K(x_i,x_j))_{1 \leq i,j \leq n}) \geq 0 \\
& \,\,\,\,\,\,\,\, \text{ for any }(x_1, \dots, x_n) \in \Lambda^n  \\
\end{align*}

\end{lemma}

\begin{proof}

(i) is clearly a consequence of (ii). We assume now that (i) is satisfied and we denote by $N$ the $\lambda^{\otimes n}$-null set of $n$-tuples $(x_1, \dots, x_n) \in \Lambda^n$ such that $F ((K(x_i,x_j))_{1 \leq i,j \leq n}) < 0$. As in remark~\ref{rq def p.s. kernels det_alpha}, we write the kernel $K$ as follows
\begin{eqnarray*}
	K(x,y) = \sum_{k=0}^\infty a_k \varphi_k(x) \overline{\psi_k}(y) = \left\langle (\sqrt{a_k} \varphi_k)_{k \in  \N} (x) | (\sqrt{a_k} \psi_k)_{k \in  \N} (y) \right\rangle
\end{eqnarray*}
where $(\varphi_k)_{k \in \N}$, $(\psi_k)_{k \in \N}$ are orthonormal basis in $L^2(\Lambda, \lambda)$, $(a_k)_{k \in \N}$ is a sequence of non-negative real number, which are the singular values of the operator $\K$ and $\left\langle . | . \right\rangle$ denote the inner product in the Hilbert space $l_2(\C)$.

As $\K$ is trace class, we have $\sum_{k=0}^\infty a_k < \infty$. Hence:
\begin{eqnarray*}
	\sum_{k=0}^\infty a_k |\varphi_k(x)|^2 < \infty \text{ and } \sum_{k=0}^\infty a_k |\psi_k(x)|^2 < \infty \text{ } \lambda\text{-a.e. } x \in \Lambda 
\end{eqnarray*}

From Lusin's theorem, there exists an increasing sequence $(A_p)_{p \in \N}$ of compact sets included in $\Lambda$ such that, for any $p \in \N$
\begin{eqnarray*}
	& (\sqrt{a_k} \varphi_k)_{k \in  \N} \text{ and } (\sqrt{a_k} \psi_k)_{k \in  \N} \text{ are continuous from } A_p \text{ to $l_2(\C)$ and } \lambda(\Lambda \backslash A_p) < \dfrac{1}{p}
\end{eqnarray*}

Therefore the kernel $K : (x,y) \mapsto \left\langle (\sqrt{a_k} \varphi_k)_{k \in  \N} (x) | (\sqrt{a_k} \psi_k)_{k \in  \N} (y) \right\rangle$ is continuous on $A_p^{\,2}$.

As $E$ is a Polish space, it can be endowed with a distance that we denote by $d$. We consider the sets
\begin{eqnarray*}
	& A'_p = \{ x \in A_p : \forall r>0, \lambda(B(x,r) \cap A_p) > 0 \} \\
	& B_{p,n} = \{ x \in A_p : \lambda (B(x,1/n ) \cap A_p) = 0 \}
\end{eqnarray*}
where $B(x,r)$ is the open ball in $E$ of radius $r$ centered at $x$ and $n$ is an integer.

Let $(x_k)_{k \in \N}$ be a sequence in $B_{p,n}$ converging to $x \in A_p$.  Then we have, when $d(x,x_k) < 1/n$,
\begin{eqnarray*}
	\lambda(B(x,1/n - d(x,x_k) \cap A_p) \leq \lambda(B(x_k, 1/n) \cap A_p) = 0
\end{eqnarray*}

Therefore $\lambda(B(x,1/n) \cap A_p) = 0$ and $x \in B_{p,n}$ : $B_{p,n}$ is closed, thus compact (as it is included in the compact set $A_p$).

The set of open balls $\{B(x,1/n) : x \in B_{p,n}\}$ is a cover of $B_{p,n}$. Then, by compactness, $B_{p,n}$ can be covered by a finite numbers of such balls. As the intersections of $A_p$ and any such a ball is a $\lambda$-null set, we get $\lambda(B_{p,n}) = 0$.\\

Hence we have: $\lambda (A'_p) = \lambda \left( A_p \backslash \cup_{n \in \N} B_{p,n} \right) = \lambda (A_p) > \lambda(\Lambda) - 1/p$.\\

Let $(x_1, \dots, x_n) \in (A'_p)^n$. If $(x_1, \dots, x_n) \notin N$, then $F ((K(x_i,x_j))_{1 \leq i,j \leq n}) \geq 0$.\\

Otherwise $(x_1, \dots, x_n) \in N$. For any $i \in \llbracket 1,n \rrbracket$ and any $r>0$, we have
\begin{eqnarray*}
 \lambda(A_p \cap B(x_i,r)) > 0 \text{, then } \lambda^{\otimes n} (A_p^n \cap B((x_1, \dots, x_n),r)) = \lambda^{\otimes n} (\prod_{i=1}^n (A_p \cap B(x_i,r))) > 0.
\end{eqnarray*}
where $B((x_1, \dots, x_n),r)$ denotes the open ball of radius $r$ centered at $x$, in $E^n$ endowed with the distance $d((x_1, \dots, x_n),(y_1, \dots, y_n)) = \max\limits_{1 \leq i \leq n} d(x_i,y_i)$.

Then, as $\lambda^{\otimes n}(N) = 0$, for any $q \in \N$, there exists $(y_1^{(q)}, \dots, y_n^{(q)}) \in A_p^n \cap B((x_1, \dots, x_n),1/q) \backslash N$ and thus $(y_1^{(q)}, \dots, y_n^{(q)})$ converge to $(x_1, \dots, x_n)$ when $q \to \infty$.

As $(y_1^{(q)}, \dots, y_n^{(q)}) \notin N$, $F ((K(y_i^{(q)},y_j^{(q)}))_{1 \leq i,j \leq n}) \geq 0$.

As $K$ is continuous on $A_p^2$ and $F$ is continuous on $\C^{n^2}$, we have that the function $(x_1, \dots, x_n) \mapsto F ((K(x_i,x_j))_{1 \leq i,j \leq n})$ is continuous on $A_p^n$. Hence we have: $F ((K(x_i,x_j))_{1 \leq i,j \leq n}) \geq 0$.

Therefore, in all cases, if $(x_1, \dots, x_n) \in (A'_p)^n$, $F ((K(x_i,x_j))_{1 \leq i,j \leq n}) \geq 0$.\\

As $(A_p)_{p \in \N}$ is an increasing sequence, it is the same for $(A'_p)_{p \in \N}$. Hence we have: $\cup_{p \in \N} (A'_p)^n = \left( \cup_{p \in \N} A'_p \right)^n$.\\

We obtain:
\begin{eqnarray*}
	F ((K(x_i,x_j))_{1 \leq i,j \leq n}) \geq 0 \text{ for any } (x_1, \dots x_n) \in \left( \cup_{p \in \N} A'_p \right)^n
\end{eqnarray*} 

As $\lambda(\Lambda \backslash \left(\cup_{p \in \N} A'_p \right)) = 0$, we finally obtain (ii) with $\Lambda' = \cup_{p \in \N} A'_p$.\\

We obtained that (i) and (ii) are equivalent conditions.\\

(i) is clearly a consequence of (iii). Assume now (ii). We will define a version $K_1$ of $K$ satisfying the condition (iii).

As $\lambda(\Lambda) \neq 0$, $\Lambda' \neq \emptyset$. We set an arbitrary $x_0 \in \Lambda'$.

For $(x,x') \in \Lambda^2$, we define, $y = x$ if $x \in \Lambda'$, $y = x_0$ if $x \in \Lambda \backslash \Lambda'$, $y' = x'$ if $x' \in \Lambda'$, $y' = x_0$ if $x' \in \Lambda \backslash \Lambda'$ and $K_1 (x,x') = K (y,y')$.

For $(x_1, \dots, x_n) \in \Lambda^n$, we define, for $1 \leq i \leq n$, $y_i = x_i$ if $x_i \in \Lambda'$ and $y_i = x_0$ if $x_i \in \Lambda \backslash \Lambda'$. Then we have, $F ((K_1(x_i,x_j))_{1 \leq i,j \leq n}) = F ((K(y_i,y_j))_{1 \leq i,j \leq n}) \geq 0$ and $K_1$ is a version of $K$ satisfying the condition (iii).

\end{proof}

\begin{remark}
Let $F_n, n \in \N,$ be continuous functions from $\C^{n^2}$ to $\C$. For any non-$\lambda-null$ compact set $\Lambda$, the condition:
\begin{itemize} [label = ]
\item[(i)] $F_n ((J_\alpha^\Lambda (x_i, x_j))_{1 \leq i,j \leq n}) \geq 0$, for any $n \in \N$ and $\lambda^{\otimes n}$-a.e. $(x_1, \dots, x_n) \in \Lambda^n$
\end{itemize}
can always be replaced by the equivalent conditions:
\begin{itemize} [label = ]
\item[(ii)] there exists a set $\Lambda' \subset \Lambda$ such that $\lambda(\Lambda \backslash \Lambda') = 0$ and $F_n ((J_\alpha^\Lambda (x_i, x_j))_{1 \leq i,j \leq n}) \geq 0$, for any $n \in  \N$ and $(x_1, \dots, x_n) \in (\Lambda')^n$.
\end{itemize}
or:
\begin{itemize} [label = ]
\item[(iii)] there exists a version of the kernel J such that $F_n ((J_\alpha^\Lambda (x_i, x_j))_{1 \leq i,j \leq n}) \geq 0$, for any $n \in  \N$ and $(x_1, \dots, x_n) \in \Lambda^n$.\\
\end{itemize}
\end{remark}

\begin{proof}
	The proof of (ii) $\implies$ (iii) is done in the same way as in Lemma~\ref{lemma p.s. kernel}.
	The other parts of the proof are a direct application of Lemma~\ref{lemma p.s. kernel}. 
\end{proof}

\begin{proof} [Proof that (i) is necessary in Theorem~\ref{th CNS alpha-determinantal si I+alpha K inversible}]

This has been mentioned by Vere-Jones in \cite{Vere-Jones alpha-permanents applications} for the multivariate binomial probability distribution, which corresponds to a determinantal process with $E$ being finite. To our knowledge, this has not been proved in other cases.

We consider the $n \times n$ matrix $1_n$, whose elements are all equal to one.\\

We have:  $\prod_{j=0}^{n-1} (1+ j \alpha) = 1 + \sum_{k=1}^{n-1} \sum_{1 \leq j_1 < \dots < j_k \leq n-1} j_1 \dots j_k \,\, \alpha^k$\\

We will show by induction on $n$ that the number of permutations in $\Sigma_n$ having $n-k$ cycles for $k \neq 0$ is $a_{n k} = \sum_{1 \leq j_1 < \dots < j_k \leq n-1} j_1 \dots j_k$:
this is true for $n=2$ and $ k=1$. Assume it is true for a given $n \in \N^*$ and for any $k \in \llbracket 1, n-1 \rrbracket$. If we consider the permutations $\sigma \in\Sigma_{n+1}$ having $n+1-k$ cycles ($0 \leq k \leq n$), we have 2 cases:\\
- either $\sigma(n+1) = n+1$: there is exactly $a_{n k}$ permutations corresponding to this case (with the convention $a_{n n} = 0$, for the case $k=n$),\\
- or $\sigma(n+1) \neq n+1$. Then, if we denote $\tau_{n+1 \, \sigma(n+1)}$ the transposition in $\Sigma_{n+1}$ that exchange $n+1$ and $\sigma(n+1)$, $\tau_{n+1 \, \sigma(n+1)} \circ \sigma$ is a permutation having $n+1$ as fixed point and $n+1 - k$ other cycles (with elements in $\llbracket 1,n \rrbracket$): there is exactly $n a_{n \, k-1}$ permutations corresponding to this case.\\
Then we have
\begin{align*}
	a_{n+1 \,\, n+1-k} & = a_{n k} + n a_{n \, k-1} \\
	  & = \sum_{1 \leq j_1 < \dots < j_k \leq n-1} j_1 \dots j_k + \sum_{\substack{1 \leq j_1 < \dots < j_{k-1} \leq n-1 \\ j_k=n}} j_1 \dots j_k \\
	  & = \sum_{1 \leq j_1 < \dots < j_k \leq n} j_1 \dots j_k
\end{align*}
which is what we expected. 

Thus: $\det\nolimits_\alpha 1_n = \prod_{j=0}^{n-1} (1+ j \alpha)$.

If $\alpha < 0$ but $-1/\alpha \notin \N$, there exists therefore $n \in \N$ such that $\det\nolimits_\alpha 1_n < 0$.\\

We suppose that there exists an $\alpha$-determinantal process with $\alpha < 0$ but $-1/\alpha \notin \N$ and kernel $K$. Then we have $\det\nolimits_\alpha (K(x_i,x_j))_{1 \leq i,j \leq n} \geq 0$ $\lambda^{\otimes n}$ -a.e. $(x_1\dots,x_n) \in E^n$.

As we exclude the case of a point process having no point almost surely and there is a sequence of compact sets $\Lambda_p$ such that $\cup_{p \in \N} \Lambda_p = E$, there exists a compact set $\Lambda \in E$ such that 
\begin{eqnarray*}
	\E(\xi(\Lambda)) = \int_\Lambda K(x,x) \lambda(dx) > 0.
\end{eqnarray*}

Applying Lemma~\ref{lemma p.s. kernel}, we get that there exist a version $K_1$ of the kernel $K$ such that $\det\nolimits_\alpha (K_1(x_i,x_j))_{1 \leq i,j \leq n} \geq 0$ for any $(x_1\dots,x_n) \in \Lambda^n$. We also have:
\begin{eqnarray*}
	\int_\Lambda K(x,x) \lambda(dx) = \int_\Lambda K_1(x,x) \lambda(dx) > 0.
\end{eqnarray*}

Hence there exists $x_0 \in \Lambda$ such that $K_1(x_0,x_0) > 0$.

For $(x_1, \dots, x_n) = (x_0, \dots, x_0)$, we get:
\begin{eqnarray*}
	\det\nolimits_\alpha (K_1(x_i,x_j))_{1 \leq i,j \leq n} = K(x_0,x_0)^n \det\nolimits_\alpha 1_n < 0
\end{eqnarray*}
which is a contradiction. Therefore if $\alpha < 0$ and an $\alpha$-determinantal process exists, then $\alpha$ must be in $ \{-1/m : m \in \N\}$.\\

\end{proof}

We consider a $d \times d$ square matrix $A$. If $n_1, \dots, n_d$ are $d$ non-negative integers, $A[n_1, \dots, n_d]$ is the $(n_1 + \dots + n_d) \times (n_1 + \dots + n_d)$ square matrix composed of the block matrices $A_{i j}$:
\begin{eqnarray*}
	A[n_1, \dots, n_d] = \begin{pmatrix} A_{1 1} & A_{1 2} & \dots & A_{1 d} \\ A_{2 1} & A_{2 2} & \dots & A_{2 d} \\ \vdots & \vdots & \ddots & \vdots \\ A_{d 1} & A_{d 2} & \dots & A_{d d}  \end{pmatrix},
\end{eqnarray*}
where $A_{i j}$ is the $n_i \times n_j$ matrix whose elements are all equal to $a_{i j}$ ($1 \leq i,j \leq d$).

\begin{lemma}
\label{lemma matrix det_alpha equiv det}
	Given a $d \times d$ square matrix $A$, the following assertions are equivalent
\begin{align*}
&	\text{(i) } \det\nolimits_{-1/m} A[n_1, \dots, n_d] \geq 0, \, \forall n_1, \dots, n_d \in \N \\
&	\text{(ii) } \det\nolimits_{-1/m} A[n_1, \dots, n_d] \geq 0, \, \forall n_1, \dots, n_d \in \{0, \dots, m \} \\
&	\text{(iii) } \det A[n_1, \dots, n_d] \geq 0, \, \forall n_1, \dots, n_d \in \N \\
&	\text{(iv) } \det A[n_1, \dots, n_d] \geq 0, \, \forall n_1, \dots, n_d \in \{0,1\}
\end{align*}

\end{lemma}

\begin{proof}

If there exists $k \in \llbracket 1,d \rrbracket$ such that $n_k > 1$, the matrix $A[n_1, \dots, n_d]$ has at least two identical rows and its determinant is null. So it is clear that (iii) and (iv) are equivalent.\\

We have:
\begin{eqnarray}
\label{formula dec dim finie det_-1/m}
	\det(I+ZA)^m = \sum_{n_1, \dots, n_d=0}^\infty \, m^{n_1 + \dots n_d} \left(\prod_{k=1}^d \dfrac{z_k^{n_k}}{n_k!}\right) \det\nolimits_{-1/m} A[n_1, \dots, n_d]
\end{eqnarray}
where $Z = \diag(z_1, \dots, z_d)$ and $z_1, \dots, z_d$ are $d$ complex numbers. It is a special case of the formula~\eqref{formula dec det Fredholm multivariate} with $\alpha = -1/m$, finite space $E = \llbracket 1,d \rrbracket$ and reference measure $\lambda$ atomic, where each point of $E$ has measure 1, $\Lambda_k = \{ k \}$, for $k \in \llbracket 1,d \rrbracket$, $\Lambda = E$. Indeed, $ZA = \sum_{k=1}^d z_k A_k$, where $A_k$ is the $d \times d$ square matrix having the same $k^{th}$ row as $A$ and the other rows with all elements equal to 0. The matrix $A$ corresponds to the operator $\K$, the matrix $A_k$ corresponds to the operator $\K_{\Lambda_k \Lambda}$. Formula~\eqref{formula dec dim finie det_-1/m} also corresponds to the one given by Vere-Jones in~\cite{Vere-Jones permanent determinant generalization}. \\

We also have for $m=1$:
\begin{eqnarray}
\label{formula dec dim finie det}
	\det(I+ ZA) = \sum_{n_1, \dots, n_d=0}^1 \, \left(\prod_{k=1}^d \dfrac{z_k^{n_k}}{n_k!}\right) \det A[n_1, \dots, n_d].
\end{eqnarray}
as $\det A[n_1, \dots, n_d] = 0$ if there exists $k \in \llbracket 1,d \rrbracket$ such that $n_k > 1$.\\

(i) is equivalent to the fact that the multivariate power series~\eqref{formula dec dim finie det_-1/m} has all its coefficients non-negative.

(iii) is equivalent to the fact that the multivariate power series~\eqref{formula dec dim finie det} has all its coefficients non-negative.\\

The power series~\eqref{formula dec dim finie det_-1/m} being the $m^{th}$ power of the power serie~\eqref{formula dec dim finie det}, if there exists  $k \in \llbracket 1,d \rrbracket$ such that $n_k > m$, the coefficient of $\prod_{k=1}^d z^{n_k}$ is null. Therefore, (i) is equivalent to (ii).\\

For the same reason, we also have that (i) is a consequence of (iii).\\

Conversely, following Vere-Jones in~\cite{Vere-Jones alpha-permanents applications}, we can show by induction on the order of the matrix $A$, that the fact that the power series~\eqref{formula dec dim finie det_-1/m} has all its coefficients non-negative implies that the power series~\eqref{formula dec dim finie det} has all its coefficient non negative.

This proves the equivalence between (i) and (iii).
	
\end{proof}

 \begin{proposition}
\label{prop CNS alpha-determinantal si I+alpha K inversible}
Let $\alpha < 0$ and $\K$ be an integral operator such that $\I+\alpha \K_\Lambda$ is invertible, for any compact set $\Lambda \subset E$. An $\alpha$-determinantal process with kernel $K$ exists iff:
\begin{eqnarray}
\label{cns det_alpha >=0}
\det\nolimits_\alpha (J_\alpha^\Lambda (x_i, x_j))_{1 \leq i,j \leq n} \geq 0 \text{, for any } n \in \N , \text { and any compact set } \Lambda \notag \\
\lambda^{\otimes n} \text{-a.e. } (x_1, \dots, x_n) \in \Lambda^n
\end{eqnarray}

Condition $\eqref{cns det_alpha >=0}$ implies that $-\dfrac{1}{\alpha} \in \N$ and $\Det(\I + \beta \K)>0 $ for any $\beta \in [\alpha,0]$.\\

\end{proposition}

\begin{proof}
We assume that there exists an $\alpha$-determinantal process $\xi$ with kernel $K$.

We already proved that it is necessary to have $-1/\alpha \in \N$.\\

By taking $d=1$ in the formula \eqref{formula pgf finite dim determinantal}, we have
\begin{eqnarray*}
	\E \left(z^{\xi(\Lambda)}\right) = \Det \left(\I + \alpha (1-z) \, \K_\Lambda\right)^{-1/\alpha}
\end{eqnarray*} 
for any compact set $\Lambda \subset E$ and $z \in (0,1]$.\\

Then $\Det \left(\I + \alpha (1-z) \, \K_\Lambda\right) > 0$ for $z \in (0,1]$, and by continuity, $\Det \left(\I + \alpha \, \K_\Lambda\right) \geq 0$. As we assumed that $\I + \alpha K_\Lambda$ is invertible, we have necessarily $\Det \left(\I + \alpha \, \K_\Lambda\right) > 0$.\\

For any non-negative function $f$, with compact support included in $\Lambda$
\begin{align*}
	\E & \left(\prod_{x \in \xi} e^{-f(x)} \right) = \Det (\I + \alpha \K [1-e^{-f}])^{-1/\alpha}\\
	& = \Det (\I + \alpha \K_\Lambda)^{-1/\alpha} \Det (\I - \alpha \J_\alpha^\Lambda e^{-f})^{-1/\alpha}\\
	& = \Det(\I + \alpha \K_\Lambda)^{-1/\alpha} \sum_{n=0}^\infty \dfrac{1}{n!} \int_{\Lambda^n} \left(\prod_{i=1}^n e^{-f(x_i)} \right) \det\nolimits_\alpha (J_\alpha^\Lambda(x_i,x_j))_{1 \leq i,j \leq n} \lambda(dx_1) \dots \lambda(dx_n)
\end{align*}

As the Laplace functional defines a.e. uniquely the Janossy density of a point process, one obtains: 
\begin{eqnarray*}
	\det\nolimits_\alpha (J_\alpha^\Lambda(x_i,x_j))_{1 \leq i,j \leq n} \geq 0 \text{ } \lambda^{\otimes n} \text{-a.e. } (x_1, \dots, x_n) \in E^n
\end{eqnarray*}

Conversely, we assume that the condition
\begin{itemize} [label =]
\item $\det\nolimits_\alpha (J_\alpha^\Lambda (x_i, x_j))_{1 \leq i,j \leq n} \geq 0$, for any $n \in \N$, $\lambda^{\otimes n}$-a.e. $(x_1, \dots, x_n) \in \Lambda^n$ and any compact set $\Lambda$.
\end{itemize}
is fulfilled. We have
\begin{eqnarray*}
	\Det(\I - \alpha z \J_\alpha^\Lambda)^{-1/\alpha} = \sum_{n=0}^\infty \dfrac{z^n}{n!} \int_{\Lambda^n} \det\nolimits_\alpha (J_\alpha^\Lambda(x_i,x_j))_{1 \leq i,j \leq n} \lambda(dx_1) \dots \lambda(dx_n)
\end{eqnarray*}

As $-1/ \alpha \in \N$, this formula is valid for any $z \in \C$. Then we obtain for $z = 1$, $\Det(\I - \alpha \J_\alpha^\Lambda)^{-1/\alpha} \geq 0$.

We also have $(\I - \alpha \J_\alpha^\Lambda)(\I + \alpha \K_\Lambda) = (\I + \alpha \K_\Lambda)(\I - \alpha \J_\alpha^\Lambda) = \I$.

Then $\Det(\I - \alpha \J_\alpha^\Lambda)>0$ and $\Det(\I + \alpha \K_\Lambda)>0$.

Thus the Janossy density is correctly defined and, on any compact set $\Lambda$ we get the existence of a point process with kernel $K$ and reference mesure $\lambda$.

Then it can be extended to the complete space $E$ by the Kolmogorov existence theorem (see Theorem 9.2.X in~\cite{Daley Vere-Jones vol II}).

\end{proof}

\begin{proof} [Proof of Theorem~\ref{th CNS alpha-determinantal si I+alpha K inversible}]

For any $m \in \N$, applying Lemma~\ref{lemma matrix det_alpha equiv det}, we have for any compact set $\Lambda$
\begin{eqnarray*}
	\det\nolimits_{-1/m} (J_{-1/m}^\Lambda (x_i, x_j))_{1 \leq i,j \leq n} \geq 0 \text{, for any } n \in \N \text{, } \text{ and any }(x_1, \dots, x_n) \in \Lambda^n
\end{eqnarray*}	
is equivalent to
\begin{eqnarray*}
	\det (J_{-1/m}^\Lambda (x_i, x_j))_{1 \leq i,j \leq n} \geq 0 \text{, for any } n \in \N \text{, } \text{ and any }(x_1, \dots, x_n) \in \Lambda^n\\
\end{eqnarray*}

Now, assume we only have
\begin{eqnarray*}
	\det\nolimits_{-1/m} (J_{-1/m}^\Lambda (x_i, x_j))_{1 \leq i,j \leq n} \geq 0 \text{, for any } n \in \N \text{, } \lambda^{\otimes n} \text{-a.e. }(x_1, \dots, x_n) \in \Lambda^n.
\end{eqnarray*}	

By lemma~\ref{lemma p.s. kernel}, for each $n \in \N$, there exists a set $\Lambda'_n \subset \Lambda$ such that $\lambda(\Lambda \backslash \Lambda'_n) = 0$ and $\det\nolimits_{-1/m} (J_{-1/m}^\Lambda (x_i,x_j))_{1 \leq i,j \leq n} \geq 0$ for any $(x_1 \dots, x_n) \in (\Lambda'_n)^n$.

If $\Lambda' = \cap_{n \in \N} \Lambda'_n$, we have $\lambda(\Lambda \backslash \Lambda') = 0$ and $\det\nolimits_{-1/m} (J_{-1/m}^\Lambda (x_i,x_j))_{1 \leq i,j \leq n} \geq 0$ for any $n \in \N$ and $(x_1 \dots, x_n) \in (\Lambda')^n$.

Then, by Lemma~\ref{lemma matrix det_alpha equiv det}, we have: $\det (J_{-1/m}^\Lambda (x_i,x_j))_{1 \leq i,j \leq n} \geq 0$, for any $n \in \N$ and $(x_1 \dots, x_n) \in (\Lambda')^n$.

Therefore, we have
\begin{eqnarray*}
	\det (J_{-1/m}^\Lambda (x_i, x_j))_{1 \leq i,j \leq n} \geq 0 \text{, for any } n \in \N \text{, } \lambda^{\otimes n} \text{-a.e. }(x_1, \dots, x_n) \in \Lambda^n.
\end{eqnarray*}	

The converse is done through a similar proof, using Lemma~\ref{lemma p.s. kernel} and~\ref{lemma matrix det_alpha equiv det}.

Thus, we obtain:
\begin{eqnarray*}
	\det\nolimits_\alpha (J_\alpha^\Lambda (x_i, x_j))_{1 \leq i,j \leq n} \geq 0 \text{, for any } n \in \N \text{, } \lambda^{\otimes n} \text{-a.e. }(x_1, \dots, x_n) \in \Lambda^n
\end{eqnarray*}	
is equivalent to
\begin{eqnarray*}
	\det (J_\alpha^\Lambda (x_i, x_j))_{1 \leq i,j \leq n} \geq 0 \text{, for any } n \in \N \text{, } \lambda^{\otimes n} \text{-a.e. }(x_1, \dots, x_n) \in \Lambda^n
\end{eqnarray*}

Theorem~\ref{th CNS alpha-determinantal si I+alpha K inversible} is then a consequence of Proposition~\ref{prop CNS alpha-determinantal si I+alpha K inversible}.

\end{proof}

\begin{proof} [Proof of Theorem~\ref{th CNS alpha-determinantal si I+alpha K_Lambda not invertible}]
We assume that there exists $\xi$ an $\alpha$-determinantal process with kernel $K$.

For $p \in (0,1)$, let $\xi_p$ be the process obtained by first sampling $\xi$, then independently deleting each point of $\xi$ with probability $1-p$. 

Computing the correlation functions, we obtain that $\xi_p$ is an $\alpha$-determinantal process with kernel $p K$.

Thus we get from Theorem~\ref{th CNS alpha-determinantal si I+alpha K inversible} that the conditions of the theorem must be fulfilled.\\

Conversely, we assume that these conditions are fulfilled. We obtain from Theorem~\ref{th CNS alpha-determinantal si I+alpha K inversible} that an $\alpha$-determinantal process  $\xi_p$ with kernel $p K$ exists, for any $p \in (0,1)$.

We consider a sequence $(p_k) \in (0,1)^\N$ converging to 1 and a compact $\Lambda$.

\begin{eqnarray*}
	\E(\exp(-t \xi_{p_k}(\Lambda)) = \Det(\I + \alpha p_k K_\Lambda (1- e^{-t}))^{-1/\alpha} \, \underset{k\to \infty}\longrightarrow \, \Det(\I + \alpha K_\Lambda (1- e^{-t}))^{-1/\alpha}
\end{eqnarray*}

As $t \mapsto \Det(\I + \alpha K_\Lambda (1- e^{-t}))^{-1/\alpha}$ is continuous in 0, $(\LL(\xi_{p_k}(\Lambda)))_{k \in \N}$ converge weakly. Thus $(\LL(\xi_{p_k} (\Lambda)))_{k \in \N}$ is tight.

$\Gamma \subset \X$ is relatively compact if and only if, for any compact set $\Lambda \subset E$, $\{ \xi(\Lambda) : \xi \in \Gamma  \}$ is bounded.

Let $(\Lambda_n)_{n \in \N}$ be an increasing sequence of compact sets such that $\cup_{n \in \N} \Lambda_n = E$.

As, for any $n \in \N$, $(\LL(\xi_{p_k} (\Lambda_n)))_{k \in \N}$ is tight, we have that, for any $\epsilon > 0$ and $n \in \N$, there exists $M_n > 0$ such that for any $k \in \N, \PP (\xi_{p_k} (\Lambda_n) > M_n) < \epsilon \, 2^{-n-1}$

Let $\Gamma = \{\gamma \in \X : \forall n \in \N, \gamma(\Lambda_n) \leq M_n \}$. It is a compact set and for any $k \in \N, \PP (\xi_{p_k} \in \Gamma^c) < \epsilon$.

Therefore, $(\LL(\xi_{p_k}))_{k \in \N}$ is tight. As $E$ is Polish, $\X$ is also Polish (endowed with the Prokhorov metric). Thus there is a subsequence of $(\LL(\xi_{p_k}))_{k \in \N}$ converging weakly to the probability distribution of a point process $\xi$.
By unicity of the distribution of an $\alpha$-determinantal process for given kernel and reference measure, $\xi$ must be an $\alpha$-determinantal process with kernel $K$, which gives the existence.

\end{proof}

\begin{lemma}
\label{lemma CNS Spec J subset [0,infty)}
Let $\J$ be a trace class self-adjoint integral operator with kernel $J$. We have
\begin{eqnarray*}
	\det(J(x_i,x_j))_{1 \leq i,j \leq n} \geq 0 \text{ for any } n\in \N, \lambda^{\otimes n} \text{-a.e.} (x_1, \dots, x_n) \in \Lambda^n
\end{eqnarray*}
if and only if
\begin{eqnarray*}
	\Spec \J \subset [0, \infty)
\end{eqnarray*}

\end{lemma}

\begin{proof}

If we assume that the operator $\J$ is positive, the kernel can be written as follows:
\begin{eqnarray*}
	J (x,y) = \sum_{k=0}^\infty a_k \varphi_k(x) \overline{\varphi_k}(y)
\end{eqnarray*}
where $a_k \geq 0$ for $k \in \N$.\\

Hence:
\begin{eqnarray*}
	\det(J(x_i,x_j))_{1 \leq i,j \leq n} \geq 0 \text{ for any } n\in \N, \text{ and any } (x_1, \dots, x_n) \in \Lambda^n\\
\end{eqnarray*}

Conversely, assume that
\begin{eqnarray*}
	\det(J(x_i,x_j))_{1 \leq i,j \leq n} \geq 0 \text{ for any } n\in \N, \lambda^{\otimes n} \text{-a.e.} (x_1, \dots, x_n) \in \Lambda^n.
\end{eqnarray*}
From formula~\eqref{formula dec det Fredholm} with $\alpha = -1$, we have then for any z $\in \C$
\begin{eqnarray}
\label{formula dec det Fredholm alpha=-1}
\Det (\I + z \J) = \sum_{n=0}^\infty \dfrac{z^n}{n!} \int_{E^n} \det (J(x_i,x_j))_{1 \leq i,j \leq n} \lambda(dx_1) \dots \lambda(dx_n).
\end{eqnarray}

As $\J$ is assumed to be self-adjoint, its spectrum is included in $\R$. Thanks to~\eqref{formula dec det Fredholm alpha=-1}, it is impossible to have an eigenvalue in $\R_-^*$, as the power series has all its coefficients real non-negative and the first coefficient ($n=0)$ is real positive. Hence $\Spec \J \subset [0, \infty)$.

\end{proof}

\begin{proof} [Proof of Corollary~\ref{corol CNS alpha-determinantal si K selfadjoint}]

We assume: $-1 / \alpha \in \N$ and $\Spec \K \subset [0, -1/\alpha]$. Then we have, as $\K$ is self-adjoint, that for any compact set $\Lambda$, $\Spec \K_\Lambda \subset [0, -1/\alpha]$. Then $\Det(\I + \beta \K_\Lambda) > 0$ for any $\beta \in (\alpha,0]$.

If $\I + \alpha K_\Lambda$ is invertible for any compact set $\Lambda \subset E$, we have $\Spec J^\Lambda_\alpha \subset [0,\infty)$ and $J^\Lambda_\alpha$ is a trace class self adjoint operator for any compact set $\Lambda$.

Then, applying Lemma~\ref{lemma CNS Spec J subset [0,infty)}, we get that
\begin{eqnarray*}
	\det(J(x_i,x_j))_{1 \leq i,j \leq n} \geq 0 \text{ for any } n \in \N, \text{ compact set } \Lambda \text{ and } \lambda^{\otimes n} \text{-a.e.} (x_1, \dots, x_n) \in \Lambda^n 
\end{eqnarray*}

Using Theorem~\ref{th CNS alpha-determinantal si I+alpha K inversible}, we get the existence of an $\alpha$-determinantal process with kernel $K$.

When there exists a compact set $\Lambda_0$ such that $\I + \alpha K_{\Lambda_0}$ is not invertible, by the same line of proof, we obtain the announced result, using Theorem~\ref{th CNS alpha-determinantal si I+alpha K_Lambda not invertible}.\\

Conversely, we assume that there exists an $\alpha$-determinantal process with kernel $K$.

Then, from Theorem~\ref{th CNS alpha-determinantal si I+alpha K inversible} or \ref{th CNS alpha-determinantal si I+alpha K_Lambda not invertible}, we get that $-1/\alpha \in \N$.

If $\I + \alpha K_\Lambda$ is invertible for any compact set $\Lambda \subset E$, we have $\Spec J^\Lambda_\alpha \subset [0,\infty)$, using Theorem~\ref{th CNS alpha-determinantal si I+alpha K inversible} and lemma~\ref{lemma CNS Spec J subset [0,infty)}. Then $\Spec K_\Lambda \subset [0,-1/\alpha) \subset [0,-1/\alpha]$, for any compact set $\Lambda$.

If there exists a compact set $\Lambda_0$ such that $\I + \alpha K_{\Lambda_0}$ is not invertible, we have $\Spec J^\Lambda_\beta \subset [0,\infty)$ for any compact set $\Lambda$ and any $\beta \in (\alpha,0)$, using Theorem~\ref{th CNS alpha-determinantal si I+alpha K_Lambda not invertible} and lemma~\ref{lemma CNS Spec J subset [0,infty)}. Then $\Spec K_\Lambda \subset [0,-1/\beta)$ for any $\beta \in (\alpha,0)$. Therefore $\Spec K_\Lambda \subset [0,-1/\alpha]$ for any compact set $\Lambda$.

As $K$ is self-adjoint, this implies in both cases that $\Spec K \subset [0,-1/\alpha]$.

\end{proof}

\begin{remark}
Using the known result in the case $\alpha = -1$ (see for example Hough, Krishnapur, Peres and Vir{\'a}g in \cite{MR2216966}) and corollary~\ref{corol equiv exist -1/m det and det process}, one obtains a direct proof of Corollary~\ref{corol CNS alpha-determinantal si K selfadjoint}.

\end{remark}

\section{Infinite divisibility}
\label{section infinite divisibility}

\begin{proof} [Proof of Theorem~\ref{th infinie divisibilité alpha-déterminantal}]
	For $\alpha < 0$, we have proved that it is necessary to have $-1/\alpha \in \N$. 
If an $\alpha$-determinantal process was infinitely divisible, with $\alpha < 0$, it would be the sum of $N$ i.i.d $\alpha N$-determinantal processes for any $N \in \N^*$, as it can be seen for the Laplace functional formula~\eqref{alpha-det Laplace functional}.
This would imply that $-1/ (N \alpha) \in \N$, for any $N \in \N^*$, which is not possible.
Therefore, an $\alpha$-determinantal process with $\alpha < 0$ is never infinitely divisible.
\end{proof}

Some charactization on infinite divisibility have also been given in~\cite{N. Eisenbaum H. Kaspi} in the case $\alpha>0$.

\begin{proof} [Proof of Theorem~\ref{th infinie divisibilité général alpha-permanental}]
For $\alpha >0$, assume that $\Det (\I + \alpha K_\Lambda) \geq 1$ and 
\begin{eqnarray*}
\sum_{\sigma \in \Sigma_n : \nu(\sigma) = 1} \, \prod_{i=1}^n J^\Lambda_\alpha(x_i, x_{\sigma(i)}) \geq 0,
\end{eqnarray*}
for any compact set $\Lambda \subset E$, $n \in \N$ and $\lambda^{\otimes n}$-a.e. $(x_1, \dots, x_n) \in \Lambda^n$. Then we have:
\begin{align*}
\sum_{\sigma \in \Sigma_n  : \nu(\sigma) = k} \, \prod_{i=1}^n J^\Lambda_\alpha(x_i, x_{\sigma(i)}) & = \sum_{\{I_1, \dots, I_k \} \atop \text{partition of } \llbracket 1,n \rrbracket } \sum_{\sigma_1 \in \Sigma (I_1), \dots , \sigma_k \in \Sigma(I_k) : \atop \nu(\sigma_1) = \dots = \nu(\sigma_k)   = 1} \, \prod_{q=1}^k \prod_{i \in I_q} J^\Lambda_\alpha(x_i, x_{\sigma_q(i)}) \\
	& = \sum_{\{I_1, \dots, I_k \} \atop \text{partition of } \llbracket 1,n \rrbracket } \prod_{q=1}^k \left( \sum_{\sigma \in \Sigma (I_q) : \atop \nu(\sigma) = 1} \, \prod_{i \in I_q} J^\Lambda_\alpha(x_i, x_{\sigma(i)}) \right) \geq 0,
\end{align*}
for any compact set $\Lambda \subset E$, $n \in \N$, $k \in \llbracket 1,n \rrbracket$ and $\lambda^{\otimes n}$-a.e. $(x_1, \dots, x_n) \in \Lambda^n$, where, for a finite set $I$, $\Sigma(I)$ denotes the set of all permutations on $I$.

Then, for any $N \in \N^*$ and any compact set $\Lambda \in E$, $\det\nolimits_{N \alpha} (J^\Lambda_\alpha (x_i,x_j)/N)_{1 \leq i,j \leq n} \geq 0$. From Theorem~\ref{th CNS alpha-permanental}, we get that there exists a $(N \alpha)$-permanental process with kernel $K/N$.
This means that an $\alpha$-permanental process with kernel $K$ is infinitely divisible.\\

Conversely, if we assume an $\alpha$-permanental process with kernel $K$ is infinitely divisible, we get the existence of a $N \alpha$-permanental process with kernel $K/N$, for any $N \in \N^*$.

From Theorem~\ref{th CNS alpha-permanental}, we have that $\Det (\I + \alpha K_\Lambda) \geq 1$ for any compact set $\Lambda \in E$.

We also have
\begin{eqnarray*}
	\dfrac{1}{(N \alpha)^{n-1}}\det\nolimits_{N \alpha} (J^\Lambda_\alpha (x_i,x_j))_{1 \leq i,j \leq n} \geq 0,
\end{eqnarray*}
 for any $N \in \N^*$, any $n \in \N$, any compact set $\Lambda \in E$ and $\lambda^{\otimes n}$-a.e. $(x_1, \dots, x_n) \in \Lambda^n$.

When N tends to $\infty$, we obtain:
\begin{eqnarray*}
\sum_{\sigma \in \Sigma_n : \nu(\sigma) = 1} \,\, \prod_{i=1}^n J^\Lambda_\alpha(x_i, x_{\sigma(i)}) \geq 0,
\end{eqnarray*}
which is the desired result.

\end{proof}

\begin{proof} [Proof of Theorem~\ref{th infinie divisibilité self-adjoint alpha permanental}]
We use the argument of Griffiths in~\cite{Griffiths} and Griffiths and Milne in~\cite{Griffiths and Milne}. Assume
\begin{eqnarray*}
\sum_{\sigma \in \Sigma_n : \nu(\sigma) = 1} \,\, \prod_{i=1}^n J^\Lambda_\alpha(x_i, x_{\sigma(i)}) \geq 0,
\end{eqnarray*}
for any $n \in \N$ and any $(x_1, \dots, x_n) \in \Lambda^n$.

The condition $J^\Lambda_\alpha (x_1,x_2) \dots J^\Lambda_\alpha(x_{n-1},x_n) J^\Lambda_\alpha(x_n,x_1) \geq 0$ is satisfied for the elementary cycles, i.e. cycles such that $J^\Lambda_\alpha(x_i,x_j) = 0$ if $i<j+1$ and ($i \neq 1$ or $j \neq n$).
Then it can be extended to any cycle by induction, using $J^\Lambda_\alpha (x_i,x_j) = J^\Lambda_\alpha (x_j,x_i)$.\\

With Lemma~\ref{lemma p.s. kernel}, we can then extend the proof to the case when 
\begin{eqnarray*}
\sum_{\sigma \in \Sigma_n : \nu(\sigma) = 1} \,\, \prod_{i=1}^n J^\Lambda_\alpha(x_i, x_{\sigma(i)}) \geq 0,
\end{eqnarray*}
for any $n \in \N$ and $\lambda^{\otimes n}$-a.e. $(x_1, \dots, x_n) \in \Lambda^n$.

\end{proof}

\begin{remark}
	Note that the argument from Griffiths and Milne in~\cite{Griffiths} and~\cite{Griffiths and Milne} is only valid for real symmetric matrices.
\end{remark}

\section*{Acknowledgements}
I express sincere thanks to Nathalie Eisenbaum for her support in the writing of this paper. I also thank the referee for his valuable comments that helped in improving the exposition.

\begin{multicols}{2}
Université Pierre et Marie Curie\\
LPMA\\
Case courrier 188\\
4, place Jussieu\\
75252 Paris Cedex 05\\
France\\
\columnbreak
\newline
\newline
Telecom ParisTech\\
LTCI\\
46, rue Barrault \\
75634 Paris Cedex 13 \\
France

\end{multicols}

\end{document}